\documentclass[10pt,letterpaper,onecolumn,oneside]{article}
\usepackage[english]{babel}
\usepackage{amsmath,amssymb,amsthm,amsfonts}
\usepackage[latin1]{inputenc}
\usepackage[T1]{fontenc}
\usepackage{anysize}
\usepackage{textpos}
\usepackage{bbm, dsfont}
\usepackage{enumerate}
\usepackage{mathrsfs}

\usepackage{bbm}
\usepackage[mathcal]{euscript}
\usepackage{graphicx}
\usepackage{pstricks}
\usepackage{pstricks-add}
\usepackage{pst-plot}
\usepackage{enumerate}
\usepackage{multicol}
\usepackage{subfigure}
\usepackage{mathrsfs}
\usepackage{amsmath,amssymb}
\usepackage{tikz}
\usepackage{stackengine}
\stackMath

\newcommand{\A}{\mathcal{A}}    
\newcommand{\B}{\mathscr{B}}    

\newcommand{\CC}{\mathbb{C}}    

\newcommand{\e}{\varepsilon}   
\newcommand{\E}{\mathcal{E}}    
\newcommand{\EE}{\mathbb{E}}     
 
\newcommand{\ind}{\mathbbm{1}} 
\newcommand{\II}{\mathbb{I}}   

\newcommand{\NN}{\mathbb{N}}
\newcommand{\N}{\mathcal{N}}   

\newcommand{\PP}{\mathbb{P}}    

\newcommand{\RR}{\mathbb{R}}    

\newcommand{\Z}{\ZZ^{d,n}}

\newcommand{\ZZ}{\mathbb{Z}}

\numberwithin{equation}{section}
\theoremstyle{plain}
\newtheorem{theorem}{Theorem}[section]
\newtheorem{lemma}[theorem]{Lemma}
\newtheorem{proposition}[theorem]{Proposition}
\newtheorem{corollary}[theorem]{Corollary}
\newtheorem*{remark}{Remark}

\theoremstyle{definition}

\usepackage{hyperref}

\title{\textbf{Branching diffusion processes and spectral properties of Feynman-Kac semigroup}}
\author{Pierre Collet\thanks{CPHT, Ecole Polytechnique, CNRS, Institut polytechnique de Paris, route de
Saclay, 91128 Palaiseau Cedex-France; E-mail:
pierre.collet@cpht.polytechnique.fr}\,,
Sylvie M\'el\'eard\thanks{CMAP, Ecole Polytechnique, CNRS, Institut polytechnique de Paris, Inria, route de
Saclay, 91128 Palaiseau Cedex-France and Institut Universitaire de France; E-mail: 
\textit{sylvie.meleard@polytechnique.edu}}\,,
Jaime San Mart\'in\thanks{CMM-DIM;  Universidad de Chile; UMI-CNRS 2807; 
Casilla 170-3 Correo 3 Santiago; Chile. email: jsanmart@dim.uchile.cl}\,.}

\begin{document}
\maketitle
\bigskip

\begin{abstract}
In this article we study the long time behavior of linear functionals of branching diffusion processes as well as the time reversal of the spinal process by means of spectral properties of the Feynman-Kac semigroup.  We generalize for this non Markovian semigroup the theory of quasi-stationary distribution (q.s.d.) and $Q$-process. The most amazing result is the identification of the law of the reversal time spinal process issued from q.s.d.  with  the $Q$-process of the Feynman-Kac semigroup.
\end{abstract}


\section{Introduction and model}
We consider a population dynamics where individuals are characterized by a trait
$x\in \mathbb{R}$ (or a position) and give birth and die in continuous time. During their life the individual
trait (or position) variations are modelled by a diffusion operator.
More precisely we consider a branching-diffusion process $(Z_{t})$ defined as follows: 
an individual with trait $x$ gives birth to a (unique) new individual with the same trait at rate $b(x)$. 
and dies with rate $d(x)$.  Changes in
the trait during the  individuals life are driven by independent diffusion processes,
accounting for infinitesimal changes of the phenotypes (or of the positions). The underlying diffusion process is  given by 
\begin{equation}
\label{diffusion}
dX_{t}= dB_{t} - a(X_{t})dt.
\end{equation}
Let us denote by $$V(x) = b(x)-d(x)$$ the growth rate for an individual with trait $x$.

Denote by ${\cal V}_{t}$ the set of individuals alive at time $t$. Then 
the process $(Z_{t}, t\ge 0)$ is the point measure valued process on the trait space $\RR$, defined for any $t$ by
$$Z_{t} = \sum_{i\in  {\cal V}_{t}} \delta_{X^i_{t}},$$
where $X^i_{t}$ is the trait of individual $i$ at time $t$. We denote the total mass of the branching diffusion process by 
$$ N_{t}=\#{\cal V}_{t}=  \langle Z_{t} ,1\rangle.$$

The empirical measure $Z_{t}$ at time $t$ on the individual traits satisfies
a semi-martingale decomposition, and is the unique solution of a
stochastic equation driven by  independent  Poisson point measures and
Brownian motions (see  for example 
Champagnat-M\'el\'eard   \cite{champagnatmeleard2007}).

The main object of this article is to study the long time behavior for the linear functionals of this process, that is 
$\,\langle Z_t,f\rangle\,$
for suitable functions $f$ and its relation with the theory of quasi-stationary distributions. 
This is done through the well known relation between the  branching diffusion process $Z$ and the 
following Feynman-Kac semigroup $(P_t)$: for any $f\in C_{b}(\RR)$, $t\ge 0$, $x\in \RR$,
\begin{equation}
\label{FK} 
P_{t}f(x) =  \EE_{{\delta_{x}}}(\langle Z_{t},f \rangle) = \EE_{x}\Big(\exp\Big(\int_{0}^t V(X_{s})ds\Big)f(X_{t})\Big).
\end{equation}
where $X$ is the drifted Brownian motion defined in \eqref{diffusion}.
So a main tool in this article is the study of the long time behavior of this semigroup. Ideas are coming from the theory 
of quasi-stationary  distributions and $Q$-processes cf. Cattiaux et al. \cite{cattiaux} (see also \cite{BCMMS}). 
A main difference with these two cited articles
is that here we do not have a strong drift from infinity, but a strong 
potential at $\pm \infty$, and because of that, another important 
source of inspiration are the ideas developed by Kac in \cite{Kac}.

Note that \eqref{FK} can be generalized to trajectories if we introduce the historical process $(L_{t}, t\ge 0)$ (cf. 
\cite{dawsonperkinsAMS}, \cite{perkins} or \cite{meleardtran_suphist}), i.e. the point measure valued process on the 
trajectories, defined for any $t$  by
\begin{equation}
\label{historical}
    L_{t} = \sum_{i\in  {\cal V}_{t}} \delta_{X^i_{.\wedge t}}.
\end{equation}

Indeed, for any $T>0$ and $F: C([0,T],\RR) \to \RR$ a bounded and continuous function and $x\in \RR$, 
\begin{equation}
\label{FK2} 
\EE_{{\delta_{x}}}(\langle L_{T},F \rangle) = \EE_{x}\Big(\exp\Big(\int_{0}^T V(X_{s})ds\Big)F\left(X^T\right)\Big),
\end{equation}
where $X^T=(X_{t\wedge T})$. These formulas/theory  come from
\cite{kurtz1997conceptual,lyons1995conceptual} further developed for
instance in \cite{bansaye_limit_2011,cloez,hardyharris3,Marguet2}. Different  and  simpler proofs are given in \cite{calvez22}.

\medskip
We denote by $\mathscr{G}$ the generator of $(X_{t})_t$, that is, $\mathscr{G}(u)=\frac12 u''-au'$, and by 
$$
\mathscr{L}=\frac12 u''-au'+Vu,
$$
the generator of $(P_t)$.
In what follows, we also introduce the following function which will play a main role in the estimates, \begin{equation}
\label{drift}\ell(x)=\int_0^x a(z)\, dz.\end{equation}

\medskip

We make the following
assumptions on $a$, $\ell$  and  $V$:  
\begin{enumerate}[H0.1)]
\item $a\in C^1(\RR)$.
\item $|\ell~|$ has at most a linear growth at $\pm \infty$,  namely there exists constants $\beta,\gamma$ so that for all $x$
it holds $|\ell(x)|\le \gamma + \beta|x|$.
\item the function $a'-a^2$ is bounded from above.
\end{enumerate}
\begin{enumerate}[H1)]
\item $V$ is continuous.
\item $V$ is bounded from above, that is $A(V):=\sup\limits_{z\in \RR} V(z)<\infty$.
\item $\lim\limits_{|x|\to \infty} V(x)=-\infty$.
\item $V$ has at least a linear decay at $\pm \infty$, that is, there exist constants $E>0, x_0\ge 0$ so that 
for all $|x|\ge x_0$ it holds
$$
V(x)\le -E |x|.
$$
\end{enumerate}
We point out that $H1$ and $H4$ imply $H2$ and $H3$, but it will be
convenient for further references to keep this redundancy.
\medskip

This article is structured as follows. In Section \ref{sec:main} we shall state the main results, 
whose proofs given in Section 4, are based on the spectral properties of $(P_t)$ developed in Section 3. 
Subsection \ref{subsec:linear} is devoted to the long time behaviour for
linear functional of $(Z_t)$, while
Section \ref{subsec:spine} studies the spinal process for the branching diffusion.
Section 5 is devoted to the specific example developed in \cite{calvez22}, 
article that motivated the present work. 
 Note that we arrive more rapidly to this result 
than in \cite{calvez22}. This example shows the efficiency of our result. 

\medskip Our results were obtained with a drifted Brownian motion but it is not a real constraint and the diffusion \eqref{diffusion} can be chosen more general. Indeed, it is well known (cf. Karatzas-Shreve \cite{karatzasshreve}) that  a diffusion
$$
dY_{t}= \sigma(Y_{t})dB_{t} - a_{0}(Y_{t})dt,
$$
where $\sigma$ is assumed to be non vanishing and $C^1$, 
  can be reduced to the case
$\sigma=1$ and drift $a$ by the image-measure transformation on the trait space
(leading to a change of measure on the measure space), with the
diffeomorphism  $y\to  G(y)=\int_{0}^y \frac{1}{\sigma(u)}du$ 
with
$\ 
a=  \left(\frac{a_{0}}{\sigma} + \frac 12 \sigma'\right)\circ G^{-1}$,
and a new potential $W=V\circ G^{-1}$. The transformed diffusion $X_t=G(Y_t)$ and its 
Feynman-Kac semigroup satisfies
$\
P^Y_t(\phi)(y)=P^X_t(\phi\circ G^{-1})(G(y))$, which permits to study the long time behaviour of the original semigroup.

\medskip
In the last years, numerous works have been developed, studying the long time behavior of non Markovian semigroups. Some of them use the Harris approach, introducing either  Doeblin's condition  (cf. Champagnat-Villemonais \cite{champagnat} and all references of the authors issued from this seminal paper) or exploiting the existence of a  Lyapounov function (cf. \cite{Marguet1}, \cite{cloezgabriel},  \cite{bansayecloezgabriel}, \cite{bansayecloezmarguet}).
Our work gives the complete spectrum of the Feynman-Kac semigroup but the main point of the paper is to translate the longtime properties of this semigroup as properties for the branching-diffusion process. In a subsequent paper, we will explore more deeply these results in the case of critical branching -diffusion processes to obtain the  time scale at which the process goes to extinction. 

\bigskip
{\bf Acknowledgments:}
This work has been supported by the Chair Mod\'elisation Math\'ematique et Biodiversit\'e of Veolia - Ecole polytechnique - 
Museum national d'Histoire naturelle - Fondation X. It is also funded by the European Union (ERC AdG SINGER, 101054787). 
Views and opinions expressed are however those of the author(s) only and do not necessarily reflect those of the 
European Union or the European Research Council. Neither the
European Union nor the granting authority can be held responsible for them. JSM also acknowledge partial support from 
BASAL ANID FB210005, and he is very grateful for the hospitality and friendly working environment of CMAP at 
the \'Ecole Polytechnique, during visits to this center where this article was initiated.

\section{Main results}
\label{sec:main}

In this section we state the main results of this article, whose proofs are based on the spectral
properties of the Feynman-Kac semigroup $(P_{t})_t$. A main tool is provided by Theorem \ref{the:mainb} in Section \ref{subsec:main_tool}.

Since $(P_t)$ is, in general, not symmetric with respect to the Lebesgue measure, it is classical (see for example \cite{cattiaux}) 
to introduce a new reference measure which makes it symmetric, the so called speed measure
$$
\rho(dy)=e^{-2\ell(y)}dy,
$$
where $\ell$ has been defined in \eqref{drift}. 

\subsection{Long time behaviour for linear functionals of $(Z_t)$}
\label{subsec:linear}

We will prove in Section 3  the following result giving the long time behavior for linear functionals of the branching diffusion process. 

\begin{theorem}
\label{thm-BP}
    Under the assumptions $H0$-$H4$,
\begin{enumerate}[(1)]
\item There exists $\lambda_0\in \RR$ such that for any $x\in \RR$, 
$$\lim\limits_{t\to \infty} -\frac{\log\left(\EE_x\left(N_t\right)\right)}{t}= \lim\limits_{t\to \infty} -\frac{\log\left(\EE_x\left(e^{\int_0^t V(X_s)}\right)\right)}{t}=\lambda_0.$$
\item   There exists a positive function $\Theta_0\in L^2(d\rho)$ defined on $\RR$ such that for all $\phi\in C(\RR)$ that has at most exponential linear growth at $\pm \infty$ and for any $x\in \RR$, it holds
$$
\lim\limits_{t\to \infty} e^{\lambda_0 t}\, \EE_x\left(\langle Z_t,\phi\rangle)\right)=
\lim\limits_{t\to \infty} e^{\lambda_0 t}\, \EE_x\left(e^{\int_0^t V(X_s)} \phi(X_t)\right)=
\Theta_0(x) \int \Theta_0(z)\phi(z)\, e^{-2\ell(z)} dz.
$$
In particular, we have the following asymptotic for the total mass 
$$
\lim\limits_{t\to \infty} e^{\lambda_0 t}\, \EE_x\left(N_t\right)=
\lim\limits_{t\to \infty} e^{\lambda_0 t}\, \EE_x\left(e^{\int_0^t V(X_s)} \right)=
\Theta_0(x) \int \Theta_0(z)\, e^{-2\ell(z)} dz,
$$
and the following ratio limit: 
$$
\lim\limits_{t\to \infty} \frac{ \EE_x\left(\langle Z_t,\phi\rangle,  N_t>0)\right)}{\EE_x\left(N_t\right)}
=\frac{\int \phi(z)\, \Theta_0(z)  e^{-2\ell(z)} dz}{\int \Theta_0(z)\, e^{-2\ell(z)} \, dz}.
$$
\end{enumerate}
\end{theorem}
In what follows we denote by $\nu$ the
probability measure on $\RR$, appearing in $(3)$, given by
$$
\nu(dx)=\frac{\Theta_0(x) e^{-2\ell(x)}}{\int \Theta_0(y) e^{-2\ell(y)}\,dy} dx,
$$
which it turns out to be the unique quasi-stationary distribution for $(P_t)$ 
(see Theorem \ref{the:qsd} and Corollary \ref{cor:uniqueqsd}). 
We notice that $\nu$ is well defined because $\Theta_0$ has exponential linear decay, of any order, at $\pm \infty$ and $\ell$ has at most a fixed linear growth at $\pm \infty$ (see Corollary \ref{cor:density} and formula \eqref{eq:boundPsik2}).

On the other hand, $\Theta_0$ is a positive eigenfunction for the 
semigroup $(P_t)$, acting in $L^2(\rho(dx))$ (see Corollary \ref{cor:density}) associated to $-\lambda_0$, that is,
for all $t>0$ it holds
$$
P_t \Theta_0=e^{-\lambda_0 t}\Theta_0.
$$

\bigskip \noindent 
{\bf Main remark}: The sign of $\lambda_0$ can be positive or negative. We will say that 
the process is supercritical (resp. subcritical or critical) if $\lambda_0<0$ (resp. $\lambda_0>0$ 
or $\lambda_0=0$).  In Section 4, it will be proved that
$$\lambda_0 > - \sup\limits_x (V(x) + 1/2(a'-a^2)(x))= -\sup\limits_x \tilde V(x).$$
Therefore if there exists a region (which can be small) such that $\widetilde V$ is positive, 
then $\lambda_0$ could be negative and the process supercritical with a mass exploding as $e^{-\lambda_0 t}$.

On the other hand, if the  potential $V$ is positive in some region, the process can be either critical, subcritical or 
supercritical depending on the drift of the underlying diffusion. For example, if $a=0$ and $V$ such that 
$\lambda_0<0$, the process is supercritical. Nevertheless  adding a suitable drift (for example a constant large enough) 
can make  the process  subcritical. The drift will push the individuals to have traits larger and larger and less and 
less  adapted. In the long run, we have extinction of the population. 
\medskip

\subsection{Time reversal of the spinal process}
\label{subsec:spine} 
In this section we will establish a result on the time reversal for the spinal process. Previous to state it and using 
the spectral properties that we will developed,
we will go deeper in the quasi-stationary (q.s.d.) properties of the semigroup $(P_t)$. In particular, we shall prove the 
existence and uniqueness of a q.s.d. distribution and the characterization of the $Q$-process. The latter will 
be related to the distribution of the reverse in time for the spinal process done in Theorem \ref{the:spine}.

\begin{theorem}
\label{the:qsd}
Assume that $H0$-$H4$ holds, then $\nu$ is a q.s.d. for the semigroup $(P_t)$,
that is, for all $t>0$ and all $\phi$ bounded and measurable functions,  it holds
$$
\int \EE_x\left(e^{\int_0^t V(X_s)ds} \phi(X_t)\right) \nu(dx)=e^{-\lambda_0 t}\int \phi(x) \nu(dx).
$$
In particular, we have
$$
\frac{\int \EE_x\left(e^{\int_0^t V(X_s)} \phi(X_t)\right) \nu(dx)}
{\int \EE_x\left(e^{\int_0^t V(X_s)}\right) \nu(dx)}= \int \phi(x) \nu(dx). 
$$
\end{theorem}
In the following result we study the domain of attraction of $\nu$. In particular, we prove that for any 
initial probability measure $\mu$
supported on $\RR$, the conditional evolution starting in $\mu$ converges to $\nu$. 

\begin{theorem} 
\label{the:domain}
Assume that $H0$-$H4$ holds and assume $\mu$ is a measure on $\RR$, so that $\int e^{-R|x|} \mu(dx)<\infty$, for some $R<\infty$. The
``conditional'' evolution $(\mu_t)_{t>0}$ given by: for all bounded and measurable $\phi$, and
$t>0$
$$
\int \phi(y) \mu_t(dy)=\frac{\int \EE_x\left(e^{\int_0^t V(X_s)ds} \phi(X_t)\right) \mu(dx)}
{\int \EE_x\left(e^{\int_0^t V(X_s)ds}\right) \mu(dx)},
$$
is a well defined probability measure on $\RR$, and 
$$
\lim\limits_{t\to \infty} \int \phi(y) \mu_t(dy)=\int \phi(y) \nu(dy).
$$
\end{theorem}

\begin{corollary} 
\label{cor:uniqueqsd}
$\nu$ is the unique q.s.d. for the semigroup $(P_t)$.
\end{corollary}

The next result gives the existence of the $Q$-process.

\begin{theorem} 
\label{the:Q-process}
Assume that $H0$-$H4$ holds, then for all $s>0$ fixed, for all 
$\phi$ bounded measurable function and for all $x$
$$
\lim\limits_{t\to \infty} \frac{\EE_x\left(e^{\int_0^t V(X_u) du} \phi(X_s)\right)}
{\EE_x\left(e^{\int_0^t V(X_u) du}\right)}=\EE_x\left(e^{\int_0^s V(X_u) du} \,
\frac{e^{\lambda_0 s} \Theta_0(X_s)}{\Theta_0(x)} \,  \phi(X_s)\right)
$$
\end{theorem}

We denote by $(Q_t)$ here the semigroup associated with the $Q$-process, that is, 
$$
Q_t(\phi)(x)=\EE_x\left(e^{\int_0^t V(X_u) du} \,\frac{e^{\lambda_0 t} 
\Theta_0(X_t)}{\Theta_0(x)} \,  \phi(X_t)\right), 
$$
In subsection \ref{sec:Q-process}, we prove Theorem \ref{the:Q-process} and 
provide the generator of this semigroup, which is given by
$$
\mathscr{L}^Q(u)=\frac12u''+ \frac{\Psi'_0}{\Psi_0} u',
$$
which corresponds to the stochastic differential equation
$$
dY_t=d B_t+ \frac{\Psi'_0(Y_t)}{\Psi_0(Y_t)}\, dt,
$$
where $\Psi_0=\Theta_0 e^{-\ell}$. This process has a unique invariant measure given by $\Psi^2_0(x)dx$.
\bigskip

For every $T>0$, the spinal process $Y^T=(Y^T_t)_{t\le T}$, associated to the branching diffusion process, is the non-homogeneous Markov process
defined as, for any $\Phi:C([0,T],\RR)\to \RR$ a bounded continuous function
$$
\EE_x(\langle L_T,\Phi \rangle)=m_T(x) \EE_x(\Phi(Y^T)),
$$
where $L_T$ is the historical process at time $T$ defined in \eqref{historical} and $m_T(x)=\EE_x(\langle L_T,1 \rangle)=\EE_x(\langle Z_T,1 \rangle)$ is the total mass at time $T$. This spinal process represents the typical trajectory explaining the history of the trait of an individual picked up at random in the population at time $T$.  This spinal process has been the object of numerous studies (cf.  for example \cite{hardyharris3}) but our motivation comes from  \cite{calvez22} (see also \cite{henry}) where the aim is to better understand the ancestors in the past explaining the distribution of  traits in the population at time $T$. Then we are interested in capturing the dynamics of the time reversal process $\stackrel{\!\!\!\longleftarrow}{Y^T}$ for the spinal process given by
$$
\stackrel{\!\!\!\longleftarrow}{Y^T_s}=Y^T_{T-s}: \, 0\le s\le T.
$$
\begin{theorem}
\label{the:spine} 
For every $T>0$, the time reversed spinal process $\stackrel{\!\!\!\longleftarrow}{Y^T}$, with initial distribution $\nu$, that is, 
$\stackrel{\!\!\!\longleftarrow}{Y^T_0}\stackrel{\mathcal{L}}{\,=\,}\nu$, is an homogeneous Markov process with
transition probability operators: for $0\le t\le T$ and $\phi$ continuous and bounded 
$$
\EE_{\nu}\left(\phi\left(\stackrel{\!\!\!\longleftarrow}{Y^T_t}\right)\right)=\langle \nu, Q_t(\phi)\rangle,
$$
where the semigroup $(Q_t)$ is given in Theorem \ref{the:Q-process} and its initial distribution is $\nu$.
\end{theorem}

\subsection{Main tool, spectral gap of $(P_t)$}
\label{subsec:main_tool}
This subsection provides the main tool we need in this article, which is a spectral gap for the semigroup $(P_t)$. In particular this will prove that the semigroup is ultracontractive  
in the sense of \cite{Davies_Simon}.
\medskip

We denote by $\Pi$  the projection onto the space generated
by $\Theta_0$, namely
$$
\Pi(g)(x)=\Theta_0(x) \int \Theta_0(y) g(y) e^{-2\ell(y)} dy=
\Theta_0(x) \frac{\int \Theta_0(y) g(y) e^{-2\ell(y)} dy}{\int \Theta_0^2(y) e^{-2\ell(y)} dy},
$$
where we will use the normalization $\int \Theta_0^2(y) e^{-2\ell(y)} dy=1$.
This projection is well defined for all functions $g$ 
for which $\Theta_0 g$ is $\rho$-integrable. 
As we will see, this includes the bounded functions, all $L^p(dx)$ and $L^p(\rho(dx))$, 
and functions with at most linear exponential
growth at $\pm \infty$. We will also see that $\Theta_0$ is a bounded $C^2$-function,  so $\Pi(g)$ fulfills the same properties
(see Section \ref{sec:estimates} for these and other properties). Note also that $P_t\Pi = e^{-\lambda_0 t}\Pi$.

In the next result we provide bounds
in $L^q(dx)$ and in $L^q(\rho(dx))$. Then, in Theorem \ref{the:mainb}, we denote by $\|\,\,\|_q$ 
the norm in either $L^q(dx)$ or $L^q(\rho(dx))$ for $q\in[1,\infty]$, and the constant $K$ 
below depends on the current measure, either $dx$ or $\rho(dx)$. We  denote by $\II$ the identity function.

\begin{theorem}  
\label{the:mainb}
Assume $H0$-$H4$ holds
\begin{enumerate}[(1)]
\item
Consider $r,q\in [1,\infty]$.
There exists a finite constant $K$ independent of $t,r,q$, 
so that, if $t_0=4(\beta+1)/E$ and $t\ge 3t_0$,  it holds
$$
\|P_t(\II-\Pi)\|_{q,r}=\sup\left\{\|P_t(\II-\Pi)(g)\|_r: \, \|g\|_q=1\right\}\le K e^{-\lambda_1 t}.
$$
In particular, for all $g\in L^q$ it holds
$$
\|e^{\lambda_0 t}P_t(g)-\Pi(g)\|_\infty\le K \|g\|_q e^{-(\lambda_1-\lambda_0)t}
$$
\item If $g$ has at most an exponential linear growth at $\pm \infty$, that is, for all $x$ it holds
$|g(x)|\le Ae^{\kappa |x|}$, for some positive constants
$A$ and $\kappa$, then there exists a constant $D=D(\kappa)$, so
that, if $t_1=4(\kappa+\beta+1)/E, 
t_0=4(\beta+1)/E$ and $t\ge 3(t_0+t_1)$, it holds
$$
\|e^{\lambda_0 t}P_t(g)-\Pi(g)\|_\infty\le A D e^{-(\lambda_1-\lambda_0)t}.
$$
\end{enumerate}
\end{theorem}

\medskip The proof of Theorem \ref{thm-BP}
 follows immediately from Theorem \ref{the:mainb} $(2)$.

\section{Basic estimates on the spectrum of $(P_t)$}
\label{sec:estimates}

As we mentioned previously, the Feynman-Kac semigroup $(P_{t})_t$ is not symmetric with respect to the Lebesgue measure
and to make it symmetric we consider the measure $\rho(dy)=e^{-2\ell(y)}dy$. 
We notice that under the hypotheses we have made on $V$, for any $t\ge 0$,  the operator $P_t$ is well 
defined for all $f$ bounded measurable, it is positive preserving, and bounded in $L^\infty$.
In Corollary \ref{cor:density}, we will prove that
the semigroup $(P_{t})_t$ has a discrete spectrum on $L^2(d\rho)$: $-\lambda_0>-\lambda_1>...>\lambda_k>...$ 
which converges to $-\infty$, and the associated set of eigenvectors $(\Theta_k)_{k\ge 0}$ is a complete set 
on $L^2(d\rho)$. In this section we will  provide some basic estimates on the eigenvalues $\lambda_k$ and eigenvectors $\Theta_k$. In particular 
we will prove that $\Theta_0$ can be chosen positive. 

\medskip 
We start the analysis of the semigroup with a Girsanov's transformation to obtain that
$$
P_t(\phi)(x)=\EE_x\left(e^{\int_0^t V(B_s) \, ds} e^{-\int_0^t a(B_s)dB_s-\frac12 \int_0^t a^2(B_s) ds}\phi(B_t)\right),
$$
where $(B_t)$ is a Brownian motion,
and from It\^o's formula, we get
$$
\ell(B_t)=\ell(x)+\int_0^t a(B_s)dB_s+\frac12 \int_0^t a'(B_s) ds,
$$
which gives that
\begin{equation}
\label{Girsanov}
\begin{array}{ll}
P_t(\phi)(x)\hspace{-0.2cm}&=\EE_x\left(e^{\int_0^t V(B_s) \, ds} e^{\ell(x)-\ell(B_t)+\frac12 \int_0^t (a'(B_s)-a^2(B_s) ) ds}\phi(B_t)\right)\\
\hspace{-0.2cm}&=e^{\ell(x)}\EE_x\left(e^{\int_0^t \widetilde V(B_s) \, ds} e^{-\ell(B_t)}\phi(B_t)\right)
=e^{\ell(x)} \widetilde P_t(e^{-\ell}\phi)(x).
\end{array}
\end{equation}
Computations show that $(\widetilde P_t)$ is also a semigroup, corresponding to the Brownian motion with a potential 
$\widetilde V=V(x)+\frac12(a'(x)-a^2(x))$, given by
$$
\widetilde P_t(g)(x)=\EE_x\left(e^{\int_0^t \widetilde V(B_s) \, ds} g(B_t)\right).
$$
 It is clear that $\widetilde V$ also
satisfies Assumptios $H1$-$H4$. We denote by $\tilde A, \widetilde{E}, \tilde x_0$ the corresponding quantities in $H2$-$H4$. 
If $a'-a^2\le M$

Notice that, we can take $\widetilde{E}=E/2$, by choosing a large $\tilde x_0$.

The semigroup $\widetilde P_t$ acts naturally on $L^2(dx)$ and if $u$ 
is a $C^2$ function, satisfying mild conditions at $\pm \infty$, then its generator is given by 
$\widetilde{\mathscr{L}}(u)=\frac12 u''+\widetilde V(x) u$. 

It is well known (see Chapter 2 Theorem 3.1 and Corollary 1 in \cite{Berezin_Shubin}), that under $H0$-$H4$ the operator 
$\widetilde{\mathscr{L}}=\frac12 \partial^2_x +\widetilde V$ has a discrete spectrum
in $L^2(dx)$ of simple eigenvalues $-\lambda_0>-\lambda_1>...>-\lambda_k>...$, which converge to $-\infty$,
and the corresponding orthonormal complete set of eigenfunctions $(\Psi_k)_{k\ge 0}$, are $C^2$ 
functions satisfying the extra decaying property: for all $k\ge 0$ and all $R>0$ there exists $C=C(R,k)$ such
that, for all $|x|\ge C$
\begin{equation}
\label{eq:1}
|\Psi_k(x)|\le e^{-R |x|}.
\end{equation} 
In particular these eigenfunctions also belong to $L^1(dx)$.
We will also see in Corollary \ref{cor:Yaglom density} that $\Psi_0$ can be chosen to be positive.

\begin{remark} We notice that $\lambda_0=\lambda_0(\widetilde V)$ could be positive $0$ or negative.  
It is enough to consider $\widehat V=\widetilde V+\lambda_0$ to have a potential 
that satisfies $H0$-$H4$ and for which $\lambda_0(\widehat V)=0$. When $\widetilde V$ is negative, then
it holds that $\lambda_0>0$. In general, we have 
$\lambda_0> -\sup\limits_x \tilde V(x)$. Indeed,  from \eqref{eq:boundPsik} below, we have
$$
\lambda_0\ge -\sup\limits_x \tilde V(x)+\frac{\pi}{e}\sup_x |\Psi_0(x)|^4.
$$ 
\end{remark}

Recall that $\mathscr{L}(u)=\mathscr{G}(u)+V(u)=\frac12 u''-au'+Vu$ is the generator of $P=(P_t)$.
We first give a rough estimate on the growth rate of the sequence $(\lambda_k)_k$.

\begin{proposition} Assume that $H0$-$H4$ holds. Then,  for all 
$k$ so that $\lambda_k>S_0=:\max\{|\widetilde V(x)|: |x|\le \tilde x_0\}$ it holds
\begin{equation}
\label{eq:growth_lambda}
\lambda_k\ge C\left(\tilde x_0,\widetilde{E}\right) k^{1/3},
\end{equation}
where $C\left(\tilde x_0,\widetilde{E}\right)=\left(\frac{\widetilde{E}^2}{8\tilde x_0+2}\right)^{1/3}$.
\end{proposition}
\begin{proof}
A basic estimate for the distribution of eigenvalues for the Schr\"odinger operator 
$\widetilde{\mathscr{L}}(u)=\frac12 u''+\widetilde V(x)u$ is obtained as follows (see Theorem 5.2 in \cite{Rozenblum_Shubin_Solomyak}).
Using the formula developed for $\Gamma(u)=-u''+Wu$, with $W=-2\widetilde V$ and whose eigenvalues are $\tau_k=2\lambda_k$, 
we get from \cite{Rozenblum_Shubin_Solomyak} that
 $$
N_W(\tau):=\#\{j: \tau_j\le \tau\}\le \int |x|(\tau-W(x))_+ \,dx+1.
$$
So, in the rest of the proof we will give an upper estimate of the integral in the right hand side.
For all $|x|\ge \tilde x_0$ it holds $\widetilde V(x)\le -\widetilde{E}  |x|$, and we choose $k$ so that $\lambda_k>S_0$.
Take now, $x_k\in \RR$ the largest, in absolute value, of the solutions to the equation
$$
2\widetilde V(x)=-\tau_k=-2\lambda_k.
$$
Clearly it holds $|x_k|>\tilde x_0$ and then $-\lambda_k=\widetilde V(x_k)\le -\widetilde{E} |x_k|$. 
We deduce that $|x_k|\le \frac1{\widetilde{E}}\lambda_k$. Also, for $|x|>|x_k|$, we have $2\widetilde V(x)<-2\lambda_k$. Then, for $\tau=2\lambda_k$, we have
$$
\begin{array}{ll}
k+1\hspace{-0.2cm}&=N_W(2\lambda_k)\le \int\limits_{-|x_k|}^{|x_k|} |x|(2\lambda_k+2\widetilde V(x))_+ \ dx +1
\le \frac2{\widetilde{E}}\lambda_k \int\limits_{-|x_k|}^{|x_k|} (\lambda_k+\widetilde V(x))_+ \ dx +1\\
\\
\hspace{-0.2cm}&= \frac2{\widetilde{E}}\lambda_k \int\limits_{|x|\le \tilde x_0} (\lambda_k+\widetilde V(x))_+ \ dx +
 \frac2{\widetilde{E}}\lambda_k \int\limits_{\tilde x_0 \le |x|\le |x_k|} (\lambda_k+\widetilde V(x))_+ \ dx +1\\
 \\
\hspace{-0.2cm}& \le  \frac{4 \tilde x_0}{\widetilde{E}} \lambda_k(\lambda_k+S_0)+
\frac2{\widetilde{E}}\lambda_k \int\limits_{ \tilde x_0\le |x|\le |x_k|} (\lambda_k+\widetilde V(x))_+ \ dx +1
 \end{array}
$$
To estimate the last integral we use, for $\tilde x_0\le |x|\le |x_k|$
$$
\lambda_k+\widetilde V(x) \le \lambda_k -\widetilde{E} |x|.
$$ 
Since the right hand side is positive, we get $(\lambda_k+\widetilde V(x))_+\le \lambda_k -E |x|$ on the range
$\tilde x_0\le |x|\le |x_k|$. We conclude that
$$
k\le \frac{8 \tilde x_0}{\widetilde{E}} \lambda^2_k+4\lambda_k\int_{0}^{\lambda_k/\widetilde{E}}
\frac1{\widetilde{E}}\lambda_k-x\,\, dx
\le \frac{8 \tilde x_0}{\widetilde{E}} \lambda^2_k+\frac2{\widetilde{E}^2}\lambda_k^3\le \frac{8 \tilde x_0+2}{\widetilde{E}^2}\lambda_k^3,
$$
and the result follows.
\end{proof}

\begin{corollary} 
\label{cor:1}
For all $t>0$ it holds $\sum\limits_{k\ge 0} e^{-\lambda_k t}<\infty$.
\end{corollary}
\begin{proof} For large $k$ we have $e^{-\lambda_k t}\le e^{-Ct\, k^{1/3}}\le k^{-2}$.
\end{proof}

The first result is the existence of densities for both $(P_t)$ and $(\widetilde P_t)$. We also provide
some a priori estimates on these eigenfunctions. 

\begin{proposition} 
\label{pro:1}
The following properties hold under $H0$-$H4$.
\begin{enumerate}[(i)]
\item for every $k$, the function $\Psi_k$ is an eigenfunction of $\widetilde P$, that is, 
for all $x,k,s>0$ it holds
\begin{equation}
\label{eq:eigen tildeP}
\Psi_k(x)=e^{\lambda_k t}\EE_x\left(e^{\int_0^t \widetilde V(B_s) ds} \Psi_k(B_s) ds\right).
\end{equation}
\item For every $k\ge 0$ and all $x$ it holds
\begin{equation}
\label{eq:boundPsik}
|\Psi_k(x)|\le\left(\frac{e}{\pi}\right)^{1/4}\left(\lambda_k+\tilde A\right)^{1/4}.
\end{equation}
On the other hand, for all $R>0$, there exist a constant $C_1=C_1(R)$ (see \ref{eq:C1R} below) and 
$t_0=t_0(R)=2R/\widetilde{E}$ such that, for all $k\ge 0$
and all $x$ it holds
\begin{equation}
\label{eq:boundPsik2}
e^{R|x|} |\Psi_k(x)|\le C_1(R) e^{\lambda_k t_0}.
\end{equation} 
Moreover, for any $t>0$ there exists a constant
$C_2(R,t)$ so that, for all $x\in \RR, k\ge 0, 0<s\le t$ it holds
\begin{equation}
\label{eq:boundPsik3}
e^{R|x|} |\Psi_k(x)|\le C_2(R,t) \frac{e^{\lambda_k s}}{s^{1/4}}.
\end{equation}
\item The semigroup $\widetilde P$ has a transition density $\tilde p(t,x,y)$ given by 
\begin{equation}
\label{eq:density tilde p}
\tilde p(t,x,y)=\sum\limits_{k\ge 0} e^{-\lambda_k t} \Psi_k(x) \Psi_k(y),
\end{equation}
where this series converges absolutely for all $t>0, x,y$, uniformly on compact sets
of $(0,\infty)\times \RR^2$. Moreover, for all $t>0, x,y$, we have the following bound
\begin{equation}
\label{eq:bound1'}
\tilde p(t,x,y)\le \sum\limits_{k\ge 0} e^{-\lambda_k t} |\Psi_k(x) \Psi_k(y)| \le \frac{e^{\tilde A t}}{\sqrt{2\pi t}}.
\end{equation}
\item the density $\tilde p$ is continuous in the three variables on $(0,\infty)\times \RR^2$.

\end{enumerate}
\end{proposition}
\begin{remark} In Section \ref{sec:1} we shall prove that $\tilde p$ is $C^{\infty,2}$ on $(0,\infty)\times \RR^2$,
improving $(iv)$ in this Proposition.
\end{remark}
\begin{proof} $(i)$ From It\^o's formula and using the stopping time $T_n=\inf\{s: |B_s|>n\}$ we have,
$$
\EE_x\left(e^{\int_0^{t\wedge T_n} \widetilde V(B_s) ds} e^{\lambda_k (t\wedge T_n)} 
\Psi_k(B_{t\wedge T_n})\right)=\Psi_k(x),
$$
where the stochastic integral part vanishes because its integrand is bounded, and the finite variation
part is $0$, because $\Psi_k$ is an eigenfunction of $\widetilde{\mathscr{L}}$. Using the dominated convergence theorem (recall that
$\Psi_k$ is continuous and bounded), we get
\begin{equation}
\label{eq:QPsi}
\EE_x\left(e^{\int_0^t \widetilde V(B_s) ds} \Psi_k(B_t)\right)=e^{-\lambda_k t}\Psi_k(x).
\end{equation}

$(ii)$ This follows from the eigenfunction equation given in $(i)$. We first consider $R=0$. 
Then, we have
$$
\begin{array}{ll}
|\Psi_k(x)|\hspace{-0.2cm}&\le e^{(\lambda_k+\tilde A) t} \EE_x(|\Psi_k(B_t)|)=e^{(\lambda_k+\tilde A) t}
\int |\Psi_k(z)| \frac{1}{\sqrt{2\pi t}} e^{-\frac{1}{2t}(z-x)^2} dz\\
\hspace{-0.2cm}&\le e^{(\lambda_k+\tilde A) t} \left(\int |\Psi_k(z)|^2 dz\right)^{1/2} \left(\frac{1}{2\pi t} 
\int e^{-\frac{1}{t}(z-x)^2} dz\right)^{1/2} = \frac{e^{(\lambda_k+\tilde A) t}}{(4\pi t)^{1/4}}.
\end{array}
$$
Notice that $\lambda_k+\tilde A>0$, otherwise we conclude that $\Psi_k\equiv 0$, by taking $t\to \infty$.
So, optimizing over $t$ we get $t^*=t^*(k)=\frac{1}{4(\lambda_k+\tilde A)}$,
which gives
$$
|\Psi_k(x)|\le\left(\frac{e}{\pi}\right)^{1/4}\left(\lambda_k+\tilde A\right)^{1/4},
$$
proving the desired property when $R=0$.

Now, assume $R>0$. We first assume $x>0$. Consider $t>0$ to be defined latter, 
which satisfies $t\widetilde{E}\ge 2R$. For $z\ge x/2\ge \tilde x_0$ one has
$$
\tilde V(z)\le -\widetilde{E} z\le -\widetilde{E} x/2.
$$
Now, we decompose the expectation in two sets $D,D^c$, where
$$
D=\left\{\min\limits_{0\le u\le t}  W_u \ge -x/2\right\},
$$
with $B_s=W_s+x$, that is, $W$ is a standard BM.
We have
$$
\begin{array}{ll}
|h^D(x)|\hspace{-0.2cm}&=e^{\lambda_k t}e^{Rx}\,\left|\EE_x\left(e^{\int_0^t \widetilde V(B_u) \, du}\,  
\Psi_k(B_t)\ind_D\right)\right|
\le e^{\lambda_k t} \EE_0\left( |\Psi_k(x+W_t)|\right)\\
\hspace{-0.2cm}&=e^{\lambda_k t} \frac{1}{\sqrt{2\pi t}}\int |\Psi_k(z)| e^{-\frac{1}{2t}(z-x)^2} dz
\le e^{\lambda_k t} \frac{1}{\sqrt{2\pi t}} \left(\int \Psi_k^2(z) dz\right)^{1/2} \left(\int e^{-\frac{1}{t}(z-x)^2} dz\right)^{1/2}\\
\hspace{-0.2cm}&=e^{\lambda_k t} \frac{1}{\sqrt{2\pi t}} (\pi t)^{1/4}=\frac{e^{\lambda_k t}}{(4\pi t)^{1/4}}.
\end{array}
$$
Recall that $\lambda_0+\tilde A>0$. Then, we get
$$
|h^D(x)|\le \frac{1}{(4\pi)^{1/4}}\frac{e^{(\lambda_k+\tilde A) t}}{t^{1/4}}.
$$
Now, we consider the expectation over $D^c$
$$
\begin{array}{ll}
|h^{D^c}(x)|\hspace{-0.2cm}&=e^{\lambda_k t}e^{Rx}\,\left|\EE_x\left(e^{\int_0^t \widetilde V(B_u) \, du}\,  
\Psi_k(B_t)\ind_{D^c}\right)\right|\le 
e^{(\lambda_k+\tilde A) t}e^{Rx}\, \EE_x(  |\Psi_k(B_t)|\ind_{D^c})\\
\hspace{-0.2cm}&\le e^{(\lambda_k+\tilde A) t} e^{Rx} (\PP_0(D^c))^{1/2} \left(\frac{1}{\sqrt{2\pi t}} \int |\Psi(z)|^2 dz\right)^{1/2}=\frac{e^{(\lambda_k+\tilde A) t}}{(2\pi t)^{1/4}} e^{Rx} (\PP_0(D^c))^{1/2}.
\end{array}
$$
By the reflection principle (see for example \cite{karatzasshreve})
$$
\begin{array}{ll}
\PP_0(D^c)\hspace{-0.2cm}&=\PP_0\left(\max\limits_{s\le t} W_u \ge x/2\right)=
\frac{2}{\sqrt{2\pi t}}\int\limits_{\frac{x}{2}}^\infty e^{-\frac{1}{2t} z^2} dz
=\frac{2}{\sqrt{2\pi}}\int\limits_{\frac{x}{2\sqrt{t}}}^\infty e^{-\frac{1}{2} u^2} du\\
\hspace{-0.2cm}&\le \frac{2}{\sqrt{2\pi}}\int\limits_{\frac{x}{2\sqrt{t}}}^\infty 
\frac{u}{\frac{x}{2\sqrt{t}}}e^{-\frac{1}{2} u^2} du
=\frac{2}{\sqrt{2\pi}}\frac{1}{\frac{x}{2\sqrt{t}}}e^{-\frac{1}{8t}x^2}.
\end{array}
$$
This implies that, for $x>8Rt+\sqrt{8t}$
$$
e^{2Rx}\PP_0(D^c)\le \frac{2e^{8tR^2}}{\sqrt{2\pi}}\frac{1}{\frac{x}{2\sqrt{t}}}e^{-\frac{1}{8t}(x-8tR)^2}\le
\frac{e^{8tR^2}}{\sqrt{\pi}}\frac{1}{\frac{x-8Rt}{\sqrt{8t}}}e^{-\frac{1}{8t}(x-8tR)^2}
\le \frac{e^{8tR^2}}{\sqrt{\pi}},
$$
which gives
$$
\begin{array}{ll}
|h^{D^c}(x)|\hspace{-0.2cm}&\le \frac{1}{(2\pi^3)^{1/4}}\frac{e^{(\lambda_k+\tilde A+4R^2) t}}{t^{1/4}}.
\end{array}
$$
Therefore, we get the bound, for $x\ge x_1=x_1(R,t)=\max\{ 2\tilde x_0, 8Rt+\sqrt{8t}\}$,
$$
e^{R x} |\Psi_k(x)|\le  \left( \frac{1}{(4\pi)^{1/4}}+\frac{1}{(2\pi^3)^{1/4}}\right) 
\frac{e^{(\lambda_k+\tilde A+4R^2) t}}{t^{1/4}}.
$$
A similar inequality holds for $x<-x_1(R,t)$.
For $|x|\le x_1$, we have, using the first part, that
$$
e^{R|x|}|\Psi_k(x)|\le \left(\frac{e}{\pi}\right)^{1/4}(\lambda_k+\tilde A)^{1/4} e^{R x_1}.
$$
Notice that $z^{1/4}\le e^z$ is valid for all $z\ge 0$, from which we deduce
\begin{equation}
\label{eq:cota2}
e^{R|x|}|\Psi_k(x)|\le  \left[\left( \frac{1}{(4\pi)^{1/4}}+\frac{1}{(2\pi^2)^{1/4}}\right) e^{(\tilde A+4R^2)t}+ 
\left(\frac{e}{\pi}\right)^{1/4}e^{R x_1}\right] \frac{e^{\lambda_k t}}{t^{1/4}}\le \Gamma(R,t)\frac{e^{\lambda_k t}}{t^{1/4}},
\end{equation}
where 
$$
\Gamma(R,t)=\left[\left( \frac{1}{(4\pi)^{1/4}}+\frac{1}{(2\pi^2)^{1/4}}\right) e^{|\tilde A+4R^2|t}+ 
\left(\frac{e}{\pi}\right)^{1/4}e^{R\, x_1(R,t)}\right].
$$
Now, we choose $t=t_0(R)=2R/\widetilde{E}$ to get
$$
e^{R|x|}|\Psi_k(x)|\le C_1(R) e^{\lambda_k t_0},
$$
where
\begin{equation}
\label{eq:C1R}
C_1(R)=\frac{\Gamma(R,t_0)}{t_0^{1/4}}=
\frac{1}{(2R/\widetilde{E})^{1/4}}\left(\frac{1}{(4\pi)^{1/4}}+\frac{1}{(2\pi^2)^{1/4}}\right) e^{|\tilde A+4R^2|\,2R/\widetilde{E}}+ 
\left(\frac{e}{\pi}\right)^{1/4}e^{R\, x_1}
\end{equation}
with $x_1=x_1(R,2R/\widetilde{E})$.

The bound \eqref{eq:boundPsik3} is obtained in the same way with
$$
C_2(R,t)=\Gamma(R,t).
$$
The result is shown (recall that $\widetilde{E}=E/2$).

$(iii)$ Take $\phi$ a continuous non negative function with compact support, in particular $\phi\in L^2(dx)$, therefore
$$
\phi(x)=\sum\limits_{k\ge 0} \Psi_k(x) \int \Psi_k(y) \phi(y) dy,
$$
where the series converges in $L^2(dx)$. 
Using that the exponential part is bounded and the dominated convergence theorem, we conclude that
$$
\int \phi(x) \EE_x\left(e^{\int_0^t \widetilde V(B_s) ds} \phi(B_t)\right)\, dx=
\sum_{k\ge 0} e^{-\lambda_k t} \int \Psi_k(x) \phi(x) dx\,  \int \Psi_k(y) \phi(y) dy. 
$$
In particular we have, for all $n\ge 1$
$$
\begin{array}{l}
\sum\limits_{k=1}^n e^{-\lambda_k t} \left(\int \Psi_k(x) \phi(x) dx\right)^2 
\le \int \phi(x) \EE_x\left(e^{\int_0^t \widetilde V(B_s) ds} \phi(B_t)\right) dx\\
\le e^{\tilde A t} \int_{\RR^2} \phi(x) \phi(y) \frac1{\sqrt{2\pi t}}e^{-\frac{(x-y)^2}{2t}} \, dx dy
\le \frac{e^{\tilde A t}}{\sqrt{2\pi t}} \int_{\RR^2} \phi(x) \phi(y)  \, dx dy.
\end{array}
$$
Using the integrability and continuity of the eigenfunctions, and taking a sequence $\phi$ converging
to $\delta_x$ we obtain the bound, for all $n,x,t>0$
$$
\sum\limits_{k=0}^n e^{-\lambda_k t} \Psi^2_k(x)\le \frac{e^{\tilde A t}}{\sqrt{2\pi t}},
$$
from where follows the point-wise convergence of the series
$$
\sum\limits_{k\ge0} e^{-\lambda_k t} \Psi^2_k(x)\le \frac{e^{\tilde A t}}{\sqrt{2\pi t}}.
$$

Then, Cauchy-Schwarz inequality shows the absolute convergence and domination for
$$
\sum\limits_{k\ge0} e^{-\lambda_k t} |\Psi_k(x)||\Psi(y)|\le \frac{e^{\tilde A t}}{\sqrt{2\pi t}}.
$$ 
This implies that $\tilde p$ is well defined and bounded on $(t_0,\infty)\times \RR^2$, for all $t_0>0$.

Consider  $\phi_2\in L^2(dx)$ a finite
combination of the eigenfunctions: $\phi_2(y)=\sum\limits_{k=0}^m a_k \Psi_k(y)$. Notice that
$\phi_2$ is continuous and bounded. As before, for all $x,t$ it holds for all $n\ge m$ (here we use the continuity of each $\Psi_k$)
$$
\begin{array}{l}
\EE_x\left(e^{\int_0^t \widetilde V(B_s) ds} \phi_2(B_t)\right)=
\sum\limits_{k=0}^m e^{-\lambda_k t} \Psi_k(x) \,  \int \Psi_k(y) \phi_2(y) dy=\\
\\
\int \sum\limits_{k=0}^n e^{-\lambda_k t} \Psi_k(x) \Psi_k(y)\,  \phi_2(y) dy=
\int \tilde p(t,x,y)\phi_2(y) dy,
\end{array}
$$ 
where the last equality follows from dominated convergence theorem. Using the density of such functions $\phi_2$ in $L^2(dx)$, 
we get that $\tilde p(t,x,\bullet)$ is a density for $\widetilde P_t(\bullet)(x)$.

$(iv)$ Assume that $(t_n,x_n,y_n)\to (t,x,y) \in (0,\infty)\times \RR$. Consider $r=\inf\{t_n; n\in \NN\}>0$
and $k_0=\inf\{k; \lambda_k>0\}$,  then for all $k\ge k_0$ we have (see \eqref{eq:boundPsik})
$$
\sup\limits_{n}e^{-\lambda_k t_n}|\Psi_k(x_n)| |\Psi_k(y_n)|\le 
e^{-\lambda_k r} \, \left(\frac{e}{\pi}\right)^{1/2}(\lambda_k+\tilde A)^{1/2}
$$
which according to Lemma \ref{cor:1} is summable. Then we can use the dominated convergence theorem (for series)
to get that
$$
\lim\limits_{n\to \infty} \sum\limits_{k\ge 0} e^{-\lambda_k t_n}\Psi_k(x_n)\Psi_k(y_n)=
\sum\limits_{k\ge 0} e^{-\lambda_k t}\Psi_k(x)\Psi_k(y),
$$
holds uniformly on compact sets of $(0,\infty)\times \RR^2$.
\end{proof}
\medskip

The following result gives a representation of the density for the semigroup $(P_t)$. 
Recall that $|\ell(x)|\le \gamma +\beta |x|$, for all $x$ (see $H0.2$). 

\begin{corollary} 
\label{cor:density}
Assume $H0$-$H4$, then 
\begin{enumerate}[(1)]
\item The set $(\Theta_k=\Psi_k e^{\ell})_k$ is an orthonormal basis of $L^2(\rho(dy))$.
\item The semigroup $(P_t)$ can be extended to $L^2(\rho(dx))$ and $L^2(dx)$. Each 
$\Theta_k=\Psi_k e^{\ell}\in L^2(\rho(dx))\cap L^2(dx)$
is an eigenvector of $P_t$, with eigenvalue $e^{-\lambda_k t}$. This semigroup has 
a symmetric density with respect to $\rho(dx)$ given by, for $t>0$
$$
r(t,x,y)=e^{\ell(x)+\ell(y)}\sum\limits_{k\ge 0} e^{-\lambda_k t} \,\Psi_k(x) \Psi_k(y)=
\sum\limits_{k\ge 0} e^{-\lambda_k t} \, \Theta_k(x) \Theta_k(y),
$$
and similarly a density with respect to $dx$
\begin{equation}
\label{eq:density_Pt}
p(t,x,y)=e^{\ell(x)-\ell(y)}\sum\limits_{k\ge 0} e^{-\lambda_k t} \, \Psi_k(x) \Psi_k(y).
\end{equation}
Also we have the bound
\begin{equation}
\label{eq:bound1}
p(t,x,y)\le e^{\ell(x)-\ell(y)}\sum\limits_{k\ge 0} e^{-\lambda_k t}\, |\Psi_k(x) \Psi_k(y)| \le 
e^{\ell(x)-\ell(y)} \frac{e^{\tilde A t}}{\sqrt{2\pi t}},
\end{equation}
and a similar bound for $r$.

\item Let $R'\ge 0$ and $s_0=2(R'+\beta)/\widetilde{E}$. Then, there exists $F=F(R')<\infty$ so that for $t>3s_0$ it holds
\begin{equation}
\label{eq:bound2}
r(t,x,y) \vee p(t,x,y)\le F e^{-\lambda_0(t-2s_0)} e^{-R'(|x|+|y|)}.
\end{equation}

Moreover, for every compact set $[t_0,t_1]\subset (0,\infty)$, both series $r,p$, converge absolutely and 
they are uniformly bounded on $[t_0,t_1] \times \RR^2$.  
If $\lambda_0\ge 0$, the same conclusion holds on $[t_0,\infty)\times \RR^2$. Also $r,p$ are continuous on 
$(0,\infty)\times \RR^2$.

\item $(P_t)_{t>0}$ can be extended to the class of functions that has at most an exponential linear growth at $\pm \infty$, and for every $\phi$ in this class and $t>0$ we have that
$$
P_t(\phi)(x)=\int r(t,x,y) \phi(y) e^{-2\ell(y)} dy=\int p(t,x,y) \phi(y) dy,
$$
is bounded in $x$, and has an exponential decay at $x=\pm \infty$, uniformly on $t\ge t_0$, for all $t_0>0$.
\end{enumerate}
\end{corollary}
\begin{proof} Only part $(3)$ needs some extra argument. Notice that, 
by Proposition \ref{pro:1} with  $R=R'+\beta$ (see \eqref{eq:boundPsik2})
$$
\begin{array}{ll}
r(t,x,y) \vee p(t,x,y)&\hspace{-0.2cm}\le e^{2\gamma+\beta(|x|+|y|)}\sum\limits_{k\ge 0} e^{-\lambda_k t} |\psi_k(x)||\psi_k(y)|
\le e^{2\gamma} (C_1(R'))^2 e^{-R'(|x|+|y|)} \sum\limits_{k\ge 0} e^{-\lambda_k (t-2s_0)}\\
&\hspace{-0.2cm}\le e^{2\gamma} (C_1(R'))^2 e^{-\lambda_0(t-2s_0)} \sum\limits_{k\ge 0} e^{-(\lambda_k-\lambda_0) s_0}   e^{-R'(|x|+|y|)}
=F(R')e^{-\lambda_0(t-2s_0)}e^{-R'(|x|+|y|)},
\end{array}
$$
where $F(R')=e^{2\gamma} (C_1(R'))^2 \sum\limits_{k\ge 0} e^{-(\lambda_k-\lambda_0) s_0}$ and the bound is shown.
The second part is obtained by using the same proposition
with the bound provided by \eqref{eq:boundPsik3}.
\end{proof}

The next Theorem is a Yaglom-type of result

\begin{corollary} 
\label{cor:Yaglom density}
Under $H0$-$H4$, for all $x,y$ it holds
$$
\lim\limits_{t\to \infty} e^{\lambda_0 t} p(t,x,y)= e^{\ell(x)-\ell(y)}\Psi_0(x)\Psi_0(y)=
e^{-2\ell(y)}\Theta_0(x)\Theta_0(y)
$$
and $\Psi_0,\Theta_0$ have constant sign, that is, we can take them to be positive. A similar result holds for $r$.
\end{corollary}

\begin{proof} The assertion follows directly  from the representation of $p(t,x,y)$. 
\end{proof}

\section{Proof of the results}
\label{sec:main results}

\subsection{Proof of Theorem \ref{the:mainb}}
\label{sec:proof the mainb}

In the proof below we do the estimates in $L^q(dx)$. Similar estimates can be done in $L^q(\rho(dx))$.

\begin{proof} 
$(1)$. Consider $g\in L^q(dx)$, then we have
$$
\begin{array}{ll}
P_t(g-\Pi(g))(x)\hspace{-0.2cm}&=e^{\ell(x)}\sum\limits_{k\ge 0} e^{-\lambda_k t} \Psi_k(x) 
\int \Psi_k(y) (g(y)-\Pi(g)(y)) e^{-\ell(y)} dy\\
\hspace{-0.2cm}&=e^{\ell(x)}\sum\limits_{k\ge 1} e^{-\lambda_k t} \Psi_k(x) \int \Psi_k(y) g(y) e^{-\ell(y)} dy.
\end{array}
$$
Then, according to Proposition \ref{pro:1} for $R=\beta+1, t_0=2R/\widetilde{E}=4R/E$,  there exists a constant
$C=e^{\gamma}C_1(R)$, where $C_1(R)$ is given in \eqref{eq:C1R}, so that for all $k,x$ it holds
$$
|\Psi_k(x)|e^{\pm \ell(x)}\le |\Psi_k(x)|e^{\gamma+\beta|x|}\le C e^{\lambda_k t_0} e^{-|x|},
$$
and, for $p$ the conjugate  of $q$
$$
\left|\int \Psi_k(y) e^{-\ell(y)} g(y) \, dy\right|\le \|\Psi_ke^{-\ell}\|_p\|g\|_q\le Ce^{\lambda_k t_0}
\left(\int e^{-p|x|} dx\right)^{1/p} \|g\|_q \le 2C e^{\lambda_k t_0} \|g\|_q.
$$
Hence, we obtain that
$$
\begin{array}{ll}
\|P_t(g-\Pi(g))\|_r \hspace{-0.2cm}&\le 2C\|g\|_q \sum\limits_{k\ge 1} e^{-\lambda_k (t-t_0)} \|\Psi_ke^{\ell}\|_r 
\le 4C^2 \|g\|_q \sum\limits_{k\ge 1} e^{-\lambda_k (t-2t_0)}\\
\\
\hspace{-0.2cm}&=e^{-\lambda_1(t-2t_0)}4C^2 \sum\limits_{k\ge 1} e^{-(\lambda_k -\lambda_1)(t-2t_0)}\|g\|_q
\le K e^{-\lambda_1 t} \|g\|_q
\end{array}
$$
where $K=4C^2 e^{2\lambda_1 t_0} \sum_{k\ge 1} e^{-(\lambda_k -\lambda_1) t_0}<\infty$.
For the second part we notice that $P_t(\Pi(g))=e^{-\lambda_0 t}\Pi(g)$, then we get
$$
\|e^{\lambda_0 t}P_t(g)-\Pi(g)\|_r\le K\|g\|_r e^{-(\lambda_1-\lambda_0)t},
$$
proving the result in this case.
\bigskip

$(2)$. The proof is similar, where this time we take $R=\kappa+\beta+1, C=e^\gamma C_1(R), t_1=2R/\widetilde{E}$ 
to bound
$$
\left|\int \Psi_k(y) e^{-\ell(y)} g(y) \, dy\right|\le A C e^{\lambda_k t_1}\int e^{-|y|} dy=2AC e^{\lambda_k t_1},
$$
and we bound $\|\Psi_k e^{\ell}\|_\infty\le e^\gamma C_1(\beta+1) e^{\lambda_k t_0}$, to get
$$
\|P_t(g-\Pi(g))\|_\infty \le A2e^{2\gamma}C_1(\kappa+\beta+1) C_1(\beta+1) 
\sum\limits_{k\ge 1} e^{-\lambda_k (t-t_1-t_0)}\le AD e^{-\lambda_1 t},
$$
where 
\begin{equation}
\label{eq:D(k)}
D=D(\kappa)= 2e^\gamma C_1(\kappa+\beta+1) C_1(\beta+1) e^{\lambda_1(t_1+t_0)} 
\sum_{k\ge 1} e^{-(\lambda_k -\lambda_1)(t_1+t_0)}.
\end{equation}
The last part of $(2)$ is shown as above.
\end{proof}

\subsection{Proof of Theorem \ref{the:qsd} and Theorem \ref{the:domain}}
\label{sec:proofs qsd domain}

\begin{proof} (Theorem \ref{the:qsd}) Assume that $\phi$ is a nonnegative function and with compact support. Then,
$$
\begin{array}{ll}
\int \EE_x\left(e^{\int_0^t V(X_s)ds} \phi(X_t)\right) \Theta_0(x) e^{-2\ell(x)} dx=
\int_{\RR^2} \Theta_0(x)  \phi(y) p(t,x,y) \phi(y) e^{-2\ell(y)}\, dy\, e^{-2\ell(x)}\, dx\\
\\
=\int_{\RR^2} \Theta_0(x) \phi(y) 
\sum\limits_{k\ge 0} e^{-\lambda_k t}\Theta_k(x)\Theta_k(y)\,e^{-2\ell(y)}\, dy\, e^{-2\ell(x)}\, dx\\
\\
=\int \Theta_0(x) \sum\limits_{k\ge 0} e^{-\lambda_k t} \Theta_k(x) (\int \Theta_k(y)  \phi(y) e^{-2\ell(y)}\, dy) e^{-2\ell(x)}\,dx\\
\\
=e^{-\lambda_0 t}\int  \Theta_0(y)\, \phi(y) \,e^{-2\ell(y)}\, dy,
\end{array}
$$
where the last three equalities follow from Fubini's Theorem, the absolute convergence of the series 
and the bound given in \eqref{eq:bound1}, the fact that $\phi$ is in $L^2(\rho(dy))$ 
and the orthonormality of $(\Theta_k)_k$. Hence, the nonnegative measures $\mu,\xi$ defined by
$$
\int \phi(y) \mu(dy):=\int \EE_x\left(e^{\int_0^t V(X_s)ds} \phi(X_t)\right) \Theta_0(x) e^{-2\ell(x)} dx
$$ 
and 
$$
\int \phi(y) \xi(dy):=e^{-\lambda_0 t}\int \phi(y)\, \Theta_0(y) e^{-2\ell(y)}\, dy,
$$
coincide for all $\phi$ nonnegative measurable and with compact support. Since $\xi$ is a finite measure (thanks to \eqref{eq:1}),
we conclude $\mu=\xi$. In particular, for all $\phi$ bounded and measurable it holds
$$
\int \EE_x\left(e^{\int_0^t V(X_s)ds} \phi(X_t)\right) \Theta_0(x) e^{-2\ell(x)} dx=
e^{-\lambda_0 t}\int \phi(x) \Theta_0(x)   e^{-2\ell(x)}\, dx
$$ 
If we take $\phi\equiv 1$ we get 
$$
\int \EE_x\left(e^{\int_0^t V(X_s)ds} \right) \Theta_0(x) e^{-2\ell(x)} dx=
e^{-\lambda_0 t}\int \Theta_0(x) e^{-2\ell(x)}\, dx,
$$ 
from which the result follows.
\end{proof}

\begin{proof} (Theorem \ref{the:domain}) For every $t_0>0$, there exists a constant $K(R,t_0)$ so that (see \eqref{eq:bound2}), 
for every $t>2t_0$ and $\phi$ bounded, we have
$$
\left|e^{\lambda_0 t} \EE_x\left(e^{\int_0^t V(X_s)ds} \phi(X_t)\right)\right|\le K(R,t_0)  \|\phi\|_\infty \, e^{-R|x|}.
$$
So, the dominated convergence theorem allows us to conclude
$$
\lim\limits_{t \to \infty} e^{\lambda_0 t} \int \EE_x\left(e^{\int_0^t V(X_s)} \phi(X_s)\right) \mu(dx)
=\int \Psi_0(z) e^{-\ell(z)} \phi(z)\, dz \int \Psi_0(x) e^{\ell(x)} \mu(dx).
$$
A similar relation holds for $\phi\equiv 1$, from where the result follows.
\end{proof}

\subsection{Proof of Theorem \ref{the:Q-process}: The generator of the $Q$-process}
\label{sec:Q-process}

\begin{proof} We first notice that for $t>s\ge 0$ it holds
$$
\EE_x\left(e^{\int_0^t V(X_u) du} \phi(X_s)\right)=
\EE_x\left(e^{\int_0^s V(X_u) du} \phi(X_s)\EE_{X_s}\left(e^{\int_0^{t-s} V(X_u) du}\right)\right)
$$
From what we have proved, for all $y$ it holds (see  Theorem \ref{thm-BP})
$$
\lim\limits_{t\to \infty} e^{\lambda_0(t-s)}\EE_{y}\left(e^{\int_0^{t-s} V(X_u) du}\right)
=e^{\ell(y)}\Psi_0(y) \int \Psi_0(z) e^{-\ell(z)}\, dz.
$$
Using Theorem \ref{the:mainb} part $(2)$ (see the definition of $t_0,t_1$ and $D=D(1)$)
and for all $(t-s)\ge 3(t_0+t_1), y\in \RR$, it holds 
$$
e^{\lambda_0(t-s)}\EE_{y}\left(e^{\int_0^{t-s}V(X_u) du}\right)
=e^{\lambda_0(t-s)} P_{t-s}(1)(y)\le C,
$$
where $C=D+\|\Pi(1)\|_\infty$.
Therefore, using the dominated convergence theorem, 
we conclude the proof.
\end{proof}
\bigskip

Recall that $(Q_t)$ is given by 
$$
Q_t(\phi)(x)=\EE_x\left(e^{\int_0^t V(X_u) du} \,\frac{e^{\lambda_0 t} e^{\ell(X_t)} \Psi_0(X_t)}{e^{\ell(x)}\Psi_0(x)} \,  \phi(X_t)\right), 
$$
whose transition kernel is then given by
\begin{equation}
\label{eq:density_Q_t}
q(t,x,y)=e^{\lambda_0 t}\frac{\Psi_0(y)}{\Psi_0(x)} e^{-\ell(x)+\ell(y)}p(t,x,y)=
e^{\lambda_0 t}\frac{\Theta_0(y)}{\Theta_0(x)}p(t,x,y)
=e^{\lambda_0 t}\frac{\Psi_0(y)}{\Psi_0(x)} \tilde p(t,x,y).
\end{equation}
It follows from the results in Section \ref{sec:1} that $q$ satisfies the following equation
$$
\partial_t q(t,x,y)=(\mathscr{L}^Q)_x^*\,q(t,x,y),
$$
where 
$$
(\mathscr{L}^Q)^*(u)=\frac12 u''-(\Psi'_0/\Psi_0) u'+
2\left(\lambda_0+\widetilde V\Psi_0+\frac12 (\Psi'_0/\Psi_0)^2\right) u.
$$
We notice that $\eta(dx)=\Psi^2_0(x) dx$ is an invariant probability measure for $(Q_t)$, which
can be checked by proving that $(\mathscr{L}^Q)^*(\Psi_0)=0$ or by using \ref{eq:eigen tildeP}. Indeed, it holds
$$
\int q(t,x,y)\Psi_0^2(x) \, dx= \Psi_0(y)  e^{\lambda_0 t} \int \tilde p(t,x,y) \Psi_0(x) \, dx=
\Psi_0(y)  e^{\lambda_0 t} \int \tilde p(t,y,x) \Psi_0(x) \, dx=\Psi_0^2(y).
$$
Notice that the generator of the $Q$-process is
$$
\mathscr{L}^Q(u)=\frac12u''+ \frac{\Psi'_0}{\Psi_0} u',
$$
which corresponds to the stochastic differential equation
$$
dY_t=d B_t+ \frac{\Psi'_0(Y_t)}{\Psi_0(Y_t)}\, dt.
$$

\subsection{Proof of Theorem \ref{the:spine}}

\begin{proof} The first thing to compute is the joint density for
$Y^T_s, Y^T_{t+s}$, where $0\le s\le t+s\le T$.
We consider $\phi,\psi$ two bounded continuous functions and 
$$
\begin{array}{ll}
m_T(x) \EE_x( \phi(Y^T_s) \Psi(Y^T_{t+s}))&\hspace{-0.2cm}
=\EE_x(e^{\int_0^T V(X_u)du} \phi(X_s) \psi(X_{t+s}))\\
&\hspace{-0.2cm}=\EE_x\left(e^{\int_0^s V(X_u)du} \phi(X_s)
\EE_{X_s}\left(e^{\int_0^t V(X_u)du}\psi(X_t)
\EE_{X_t}\left(e^{\int_0^{T-(t+s) V(X_u)du}}\right)\right)\right).
\end{array}
$$
From here we deduce that the joint density, with respect to $dy\,dz$ is
\begin{equation}
\label{eq:density_s_t}
\PP_x( Y^T_s\in dy, Y^T_{t+s}\in dz)=p(s,x,y)p(t,y,z)
\frac{m_{T-(t+s)}(z)}{m_T(x)}\, dy\, dz
\end{equation}
In particular, if the initial distribution of $Y^T_0$ is 
$e^{\lambda_0 T}m_T(x) \nu(dx)$, we get the distribution of $Y^T_s$ has density
with respect to Lebesgue measure $dy$ given by (take $T=t+s$ in the previous equality)
$$
\begin{array}{ll}
e^{\lambda_0 T}\int \Theta_0(x) e^{-2\ell(x)} p(s,x,y)p(T-s,y,z)dxdz&\hspace{-0.2cm}
=e^{\lambda_0 (T-s)}\Theta_0(y) e^{-2\ell(y)} \int p(T-s,y,z)dz\\
&\hspace{-0.2cm}=e^{\lambda_0 (T-s)}\Theta_0(y) e^{-2\ell(y)}m_{T-s}(y).
\end{array}
$$
Here we have used that $\Theta_0 e^{-2\ell}$ is a left eigenvector 
for $P_s$ with eigenvalue $e^{-\lambda_0 s}$
(this is equivalent to say that $\nu$ is a q.s.d. for $(P_t)$).
So, we get
$$
\EE_{e^{\lambda_0 T}m_T\nu}(Y^T_s\in dy)=e^{\lambda_0 (T-s)}m_{T-s}(y)\nu(dy).
$$
This implies that $Y^T_T$ has distribution $\nu$, if the initial 
distribution is $e^{\lambda_0 T}m_T(x)\nu(dx)$. 
Now we compute the conditional density for the reversed process, 
when the initial distribution is
$\PP\left(\stackrel{\!\!\!\longleftarrow}{Y^T_0}\in dy\right)=\nu(dy)$
$$
\PP\left(\,\,\stackrel{\hspace{-0.4cm}\longleftarrow}{Y^T_{u+r}}\in dy \,\,\, \Big | \, 
\stackrel{\!\!\!\longleftarrow}{Y^T_{r}}=z\right)
=\PP\left(Y^T_{T-(u+r)}\in dy \,\Big|\, Y^T_{T-r}=z\right).
$$
We put $s=T-(u+r)$ and $t+s=T-r$, that is $t=u$ in \eqref{eq:density_s_t} and using the 
marginal density for $\stackrel{\!\!\!\longleftarrow}{Y^T_{r}}$
\begin{equation}
\label{eq:density_reversed}
\begin{array}{ll}
\hspace{-0.4cm}\PP\left(\,\,\stackrel{\hspace{-0.4cm}\longleftarrow}{Y^T_{u+r}}\in dy \,\,\,
\Big |\, \stackrel{\!\!\!\longleftarrow}{Y^T_{r}}=z\right)&\hspace{-0.2cm}=
\frac{1}{e^{\lambda_0 r}m_{r}(z)\theta_0(z)e^{-2\ell(z)}}
\int p(s,x,y)p(u,y,z)\frac{m_{r}(z)}{m_T(x)} m_T(x)
e^{\lambda_0 T}\Theta_0(x)e^{-2\ell(x)}dx\\
\\
&\hspace{-0.2cm}=e^{\lambda_0(T-r)} 
\frac{p(u,y,z)}{\Theta_0(z)} e^{2\ell(z)} 
\int p(s,x,y)\Theta_0(x)e^{-2\ell(x)}dx\\
\\
&\hspace{-0.2cm}=e^{\lambda_0(T-r-s)} 
\frac{\Theta_0(y)}{\Theta_0(z)} e^{2\ell(z)-2\ell(y)} p(u,y,z)=e^{\lambda_0 u} \frac{\Theta_0(y)}{\Theta_0(z)} p(u,z,y)\\
\\
&\hspace{-0.2cm}=q(u,z,y).
\end{array}
\end{equation}
Here, we have used again that $\Theta_0 e^{-2\ell}$ is a left eigenvector of $P_s$ with eigenvalue $e^{-\lambda_0s}$ and 
the symmetry relation $e^{2\ell(z)-2\ell(y)} p(u,y,z)=p(u,z,y)$ (see \eqref{eq:density_Pt}). 
The last equality in \eqref{eq:density_reversed} is given in \eqref{eq:density_Q_t} . The result follows.
\end{proof}

\section{Example : Case $V(x)=1-x^2/2, dX_t=\sigma(dB_t-c dt)$}

We develop here our approach in the case studied in \cite{calvez22}, where $b\equiv 1$, $d(x)=x^2/2$ and the underlying diffusion process is a drifted Brownian motion with variance $\sigma^2$ and drift $c \sigma$, $\sigma$ being the mutation scale parameter. Indeed in this case, we will have explicit computations involving Hermite polynomials. 
\medskip 

We define as in the  above framework the semigroup
$$
P_t(\phi)(x) =\EE_x\Big(e^{\int_0^t  V(X_s)ds} \phi(X_t)\Big)=\EE_0\Big(e^{\int_0^t V(x+\sigma(B_s -cs))ds} \phi(x+\sigma(B_t -ct))\Big).
$$
Using Girsanov's transform (see \eqref{Girsanov}),  we obtain
$$
P_t(\phi)(x) = \EE_0\Big(e^{\int_0^t (V(x+\sigma W_s)- c^2/2)ds} e^{-cW_t}\phi(x+\sigma W_t)\Big),
$$
where $W$ is a new Brownian motion. 
Writing $\sigma W_t=\widehat B_{\sigma^2 t}$ with another Brownian motion $\widehat B$, we obtain
$$
P_t(\phi)(x) = e^{cx/\sigma}\EE_0\left(\exp\left(\frac{1}{\sigma^2}\int_0^{\sigma^2 t} \hspace{-0.4cm} V(x+\widehat B_{u})- c^2/2\, du
-\frac{c}{\sigma}\left(x+ \widehat B_{\sigma^2 t}\right)\right)
\,\phi\left(x+ \widehat B_{\sigma^2 t}\right)\right).
$$
Writing $V(x) = 1-x^2/2$, we finally obtain that 
$$ P_t(\phi)(x)=e^{(1-c^2/2)t} e^{cx/\sigma} R_{\sigma^2 t} (\varphi)(x),
$$
where $\varphi(y)=e^{-\frac{c}{\sigma} y} \phi(y)$, 
and
$$
R_t(g)(y)=\EE_y\Big(e^{-\frac{1}{2\sigma^2} \int_0^t (\widehat B_u)^2 du} g(\widehat B_t)\Big).
$$
Notice that the generator of the semigroup $(R_t)$ is given for suitable functions $v$ by
$$
{\cal L}_R\, v= \frac{1}{2} \partial_x^2 v - \frac{x^2}{2 \sigma^2} v.
$$
The eigenvectors and eigenvalues associated to this generator satisfy
$$
\frac12 \Psi_k''(x)- \frac{1}{2\sigma^2} x^2 \Psi_k(x)=-\lambda_k \Psi_k(x).
$$
Consider $g_k(x)=\Psi_k(\sqrt{\sigma}x)$ which satisfies the equation
$$
\frac12 g''(x)-\frac12 x^2 g(x)=-\sigma\lambda_k g(x)=-\nu_k g(x),
$$
whose solutions are the Hermite functions
$H_k(x) e^{-\frac12 x^2 }$, and associated eigenvalues are $\nu_k=k+1/2$, 
with $H_k$ the Hermite polynomials of degree $k$ (see \cite{Szego} page 106). So, 
$$
\Psi_k(x)=g_k(x/\sqrt{\sigma})=C_k H_k(x/\sqrt{\sigma}) e^{-\frac{x^2}{2\sigma}}, \, \lambda_k=\frac{k+1/2}{\sigma}
$$
and $C_k$ is a normalizing constant, in order to have $\int \Psi_k^2(x) dx=1$.
Since $\int H_k^2(x) e^{-x^2} \, dx=2^k k! \sqrt{\pi}$, we conclude that
$$
C_k=\sigma^{-1/4}\frac{1}{\sqrt{2^k k! \sqrt{\pi}}}
$$
In particular
$$
\Psi_0(x)=\left(\frac{1}{\sigma \pi}\right)^{1/4}e^{-\frac{x^2}{2\sigma}}.
$$
On the other hand, Cramer inequality (see \cite{Erdelyi} page 208 and \cite{Indritz}) gives the following refinement of \eqref{eq:boundPsik} in this context, for all $x, k$
$$
|\Psi_k(x)|\le \left(\frac{1}{\sigma\pi}\right)^{1/4}.
$$
For the moment we have the asymptotic
$$
\lim\limits_{t\to \infty}e^{\frac{1}{2\sigma}t} e^{\frac{c}{\sigma}x} R_t(e^{-\frac{c}{\sigma}\bullet}\phi(\bullet))(x)
=\Psi_0(x)e^{\frac{c}{\sigma}x} \int \Psi_0(y)e^{-\frac{c}{\sigma}y} \phi(y) dy.
$$
This translates into the original semigroup, with potential $V(x)=1-\frac12 x^2$, $dX_t=\sigma(dB_t-cdt), X_0=x$
and $\lambda_0=1-\frac{c^2}2-\frac\sigma2$:
$$
\lim\limits_{t\to \infty} e^{\lambda_0 t}\EE_y\left(e^{\int_0^t V(X_s) \, ds} \phi(X_t) \right)=e^{-\frac{x^2}{2\sigma}+\frac{c}{\sigma}x}
\!\int \!\!\frac{1}{\sqrt{\pi\sigma}} e^{-\frac{y^2}{2\sigma}-\frac{c}{\sigma}y} \phi(y) dy
=e^{\frac{c^2}{\sigma}}e^{-\frac1{2\sigma}(x-c)^2}\!\!\!\int \!\! \frac{1}{\sqrt{\pi\sigma}} e^{-\frac1{2\sigma}(y+c)^2} \phi(y) dy.
$$
In particular for $\phi=1$, we get for $m_t(x)=\EE_y\left(e^{\int_0^t V(X_s)\, ds}  \right)$
$$
\lim\limits_{t\to \infty} e^{\lambda_0 t} m_t(x)= 
\sqrt{2}\,e^{\frac{c^2}{\sigma}}e^{-\frac1{2\sigma}(x-c)^2}.
$$
Notice that an explicit form for the function  $m_t$ has been given in \cite{calvez22}, providing
another way to calculate this limit.

Notice that $\ell(x)=\frac{c}{\sigma}$ and recall that $\rho(dx)=e^{-2\ell(x)}dx$. 
The principal right eigenvector for $(P_t)$, normalized to have $L^2(\rho(dx)$-norm equal to 1
is
$$
\Theta_0(x)=e^{\frac{c^2}{2\sigma}}\frac{e^{-\frac{(x-c)^2}{2\sigma}}}{(\pi \sigma)^{1/4}}.
$$
The unique q.s.d. is given by 
$$
d\nu(x)=\frac{1}{\sqrt{2\pi \sigma}}e^{-\frac{(y+c)^2}{2\sigma}}.
$$
The $Q$ process is given by the stochastic differential equation
$$
d\hat Y_t=-\sigma \hat Y_t+\sigma dB_t.
$$
whose invariant measure is $\Psi_0^2(x)= \frac{1}{\sqrt{\pi \sigma}} e^{-\frac{x^2}{\sigma}}$.
\bigskip

\appendix
\section*{Appendix A: The heat equation for $p$}
\label{sec:1}

In this appendix we shall prove that $p$ is a positive solution of the heat equation 
$\partial_t p(t,x,y)=\mathscr{L}_y p(t,x,y)$. To this end we first start with $\tilde p$.

We first notice that the series $\sum\limits e^{-\lambda_k t} \Psi_k(x) \Psi_k''(y)$ converges absolutely and uniformly
on compact sets in $(0,t)\times \RR^2$. Indeed, consider $K=[t_0,t_1]\times [-c,c]^2$, where $0<t_0<t_1<\infty
,0<c<\infty$.  For $(t,x,y)\in K$, we have
$$
|\Psi_k''(y)|=|2\widetilde V(y)+2\lambda_k| |\Psi_k(y)|\le 
2\left(\max\limits_{|y|\le c} |\widetilde V(y)|+|\lambda_k|\right) |\Psi_k(y)|,
$$ 
which implies that for some finite constant $C$ (see Proposition \ref{pro:1} formula \eqref{eq:boundPsik})
$$
\sum\limits_{k\ge 0} e^{-\lambda_k t} |\Psi_k(x)||\Psi_k''(y)| \le 
C \sum\limits_{k\ge 0} e^{-\lambda_k t}
\left(\max\limits_{|y|\le c} |V(y)|+|\lambda_k|\right)(\lambda_k+\tilde A)^{1/2}.
$$
The series in the right hand side is converging. This shows
the claim and moreover 
$$
\zeta(t,x,y):=\sum\limits_{k\ge 0} e^{-\lambda_k t} \Psi_k(x) \Psi_k''(y),
$$
is well defined and continuous in $K$. We now need to control the growth of $(\Psi_k'(x^*))_k$ for some
$x^*$ fixed.
\begin{lemma} There exist two finite constants $A_1\ge 0,A_2\ge 0$ and 
$x^*\in \RR$ such that, for all $k$, we have
\begin{equation}
\label{eq:growthPsi'k}
|\Psi_k'(x^*)|\le A_1+A_2 \lambda_k.
\end{equation}
\end{lemma}
\begin{proof}
It is well known that each $\Psi_k$ has exactly $k$ zeros (see Theorem 3.5 in \cite{Berezin_Shubin}). 
Then, $u=\Psi_2$  has two zeros $x_1<x_2$. We conclude, there exists
$x^*\in(x_1,x_2)$ such that $u'(x^*)=0$, and therefore $|u(x^*)|>0$. Now, we consider $v=\Psi_k$, 
for $k\ge 3$ and the Wronskian 
$$
W(x)=\frac12(u'v-uv').
$$
Since $W'(x)=\frac12(u''v-uv'')=(\lambda_k-\lambda_2) uv$, we get
$$
\begin{array}{ll}
W(x^*)-W(x_1)\hspace{-0.2cm}&=\frac12 (u'(x^*)v(x^*)-u(x^*)v'(x^*))-\frac12(u'(x_1)v(x_1)-u(x_1)v'(x_1))\\
\hspace{-0.2cm}&=-\frac12 u(x^*)v'(x^*)-\frac12 u'(x_1)v(x_1)=(\lambda_k-\lambda_2)\int_{x_1}^{x^*} u(z)v(z) \ dz,
\end{array}
$$
and therefore
$$
\begin{array}{ll}
|v'(x^*)| |u(x^*)|\hspace{-0.2cm}&\le |v(x_1)||u'(x_1)|+2(\lambda_k-\lambda_2)\int\limits_{-\infty}^\infty |u(z)v(z)| \, dz\\
\hspace{-0.2cm}&\le  \left(\frac{e}{\pi}\right)^{1/4}|u'(x_1)|\, (\lambda_k+\tilde A)^{1/4}+2(\lambda_k-\lambda_2)
\left(\int\limits_{-\infty}^\infty |u(z)|^2\, dz\right)^{1/2} 
\left(\int\limits_{-\infty}^\infty |u(z)|^2\, dz\right)^{1/2}\\
\hspace{-0.2cm}&\le \left(\frac{e}{\pi}\right)^{1/4}  |u'(x_1)| \, (\lambda_k+\tilde A)^{1/4}+2(\lambda_k-\lambda_2)\\
\hspace{-0.2cm}&\le \left(\frac{e}{\pi}\right)^{1/4} |u'(x_1))| \,  (\lambda_k+\tilde A)^{1/4}+2|\lambda_2|+ 2\lambda_k.
\end{array}
$$
Consider $k_0=\inf\{k: \lambda_k+\tilde A\ge 1\}$, then for $k\ge k_0$,  we have 
$$
|\Psi_k'(x^*)|\le D_1+A_2\lambda_k,
$$
with $D_1=\frac{\left(\frac{e}{\pi}\right)^{1/4}|u'(x_1))| \tilde A + 2|\lambda_2|}{|u(x^*)|}$ and $A_2=
\frac{\left(\frac{e}{\pi}\right)^{1/4}|u'(x_1))| + 2}{|u(x^*)|}$.
The result follows by taking 
$$
A_1=D_1\vee \max\{|\Psi_j'(x^*)|+A_2|\lambda_j|: 0\le j\le k_0-1\}.
$$

\end{proof}
From the previous Lemma, the series
$$
S_1(t,x):=\sum\limits_{k\ge 0} e^{-\lambda_k t} \Psi_k(x) \Psi_k'(x^*), \, 
S_2(t,x,y):=\sum\limits_{k\ge 0} e^{-\lambda_k t} \Psi_k(x) (\Psi_k'(x^*) y+\Psi_k(0)) 
$$ 
are absolutely convergent, uniformly on $K$ and continuous there. 
\medskip

The DCT allow us to show that
$$
\begin{array}{l}
\int_{0}^y\int_{x^*}^z \zeta(t,x,w) \, dw\, dz+\int_0^y S_1(t,x) \, dz+ \tilde p(t,x,0)\\
\\
=\sum\limits_{k\ge 0} e^{-\lambda_k} \Psi_k(x)
\left(\int_0^y\int_{x^*}^z \Psi''_k(w) dw\, dz+\Psi'_k(x^*)y+\Psi_k(0)\right)\\
\\
=\tilde p(t,x,y).
\end{array}
$$ 
This shows that $\tilde p(t,x,y)$ is twice continuously differentiable in $y$ and 
$$
\begin{array}{l}
\partial_y \tilde p(t,x,y)=\sum\limits_{k\ge 0} e^{-\lambda_k t} \Psi_k(x) \Psi_k'(y), \,
\partial^2_y \tilde p(t,x,y)=\sum\limits_{k\ge 0} e^{-\lambda_k t} \Psi_k(x) \Psi_k''(y);\\
\partial_x \tilde p(t,x,y)=\sum\limits_{k\ge 0} e^{-\lambda_k t} \Psi_k'(x) \Psi_k(y), \,
\partial^2_x \tilde p(t,x,y)=\sum\limits_{k\ge 0} e^{-\lambda_k t} \Psi_k''(x) \Psi_k(y)
\end{array}
$$
The same technique shows that $\tilde p\in C^{\infty,2}((0,\infty)\times\RR^2)$ and
$$
\partial_t \tilde p(t,x,y)=\sum\limits_{k\ge 0}-\lambda_k e^{-\lambda_k t} \Psi_k(x)\Psi_k(y).
$$
Now, we have
$$
\begin{array}{ll}
\partial_t \tilde p(t,x,y)\hspace{-0.2cm}&=\sum\limits_{k\ge 0}-\lambda_k e^{-\lambda_k t} \Psi_k(x)\Psi_k(y)=
\sum\limits_{k\ge 0} e^{-\lambda_k t} \Psi_k(x)\left(\frac12 \Psi_k''(y)+\widetilde{V}(y) \Psi_k(y)\right)\\
\hspace{-0.2cm}&=\frac12 \partial^2_y \tilde p(t,x,y)+\widetilde{V}(y)\tilde p(t,x,y)=\widetilde{\mathscr{L}}_y \, \tilde p(t,x,y)=
\widetilde{\mathscr{L}}_x\, \tilde p(t,x,y)
\end{array}
$$

Now, we can apply the relation to $p(t,x,y)=e^{\ell(x)\ell(y)}\tilde p(t,x,y)$ to conclude that
$$
\begin{array}{l}
\partial_t p(t,x,y)=e^{\ell(x)-\ell(y)}\partial_t \tilde p(t,x,y),\\
\partial_x p(t,x,y)=e^{\ell(x)-\ell(y)}\left[\partial_x \tilde p(t,x,y)+a(x)\tilde p(t,x,y)\right]\\ 
\partial^2_x p(t,x,y)=e^{\ell(x)-\ell(y)}\left[\partial^2_x \tilde p(t,x,y)+2a(x)\partial_x \tilde p(t,x,y)
+a^2(x)\tilde p(t,x,y)+a'(x)\tilde p(t,x,y)\right].
\end{array} 
$$
So, we get
$$
\partial_t p(t,x,y)=\mathscr{L}_y p(t,x,y),
$$
where $\mathscr{L}(u)=\frac12 u''-au'+Vu$, and $\frac12 u''-au'$ is the generator of $(X_t)$. Also, we get
$$
\partial_t p(t,x,y)=\mathscr{L}^*_x p(t,x,y),
$$
where $\mathscr{L}^*(u)=\frac12 u''+au'+a'u+Vu$.

\begin{corollary} 
\label{cor:2}
Both $p$ and $\tilde p$ satisfy Harnack's inequality
\end {corollary}

\end{document}